\documentclass[12pt]{amsart}

\usepackage{amscd,amsthm}
\usepackage{latexsym}
\usepackage{amsmath,amstext,amsthm,amsfonts,amssymb}
\usepackage[dvips]{graphicx}
\usepackage{anysize}
\marginsize{22mm}{18mm}{17mm}{30mm}
\usepackage{newlfont}
\usepackage[ansinew]{inputenc}
\usepackage{graphpap}

\usepackage{mathrsfs}
\usepackage{enumerate} 
\usepackage{xcolor}
\usepackage[active]{srcltx}
\usepackage[colorlinks=true,linkcolor=red,urlcolor=black]{hyperref}

\theoremstyle{plain} 
\newtheorem{teo}{Teorema }[section]
\newtheorem{maintheorem}{Theorem}

\newtheorem{mainpropo}{Proposition}

\newtheorem{coro}[teo]{Corollary}
\newtheorem{lema}[teo]{Lemma}

\newtheorem{defi}[teo]{Definition}

\title{ Topological Entropy and Metric Entropy for Regular Impulsive Semiflows}

\author{Nelda Jaque \and  Bernardo San Mart\'in }
\address{
Nelda Jaque, 
Departamento de Matem\'aticas, 
Universidad de Atacama, 
Copayapu 485, Casilla 240,
Copiap\'o, Chile.}
\email{njaquetamblay@gmail.com}
\address{
Bernardo San Mart\'in, 
Departamento de Matem\'aticas, 
Universidad Cat\'olica del Norte, 
Avenida Angamos 0610, Casilla 1280,
Antofagasta, Chile.}
\email{sanmarti@ucn.cl}
\thanks{  The author was partially supported by
 FONDECYT Grant 1181183 through CONICYT (Chile).}

\begin{document}

\maketitle

\begin{abstract}
In \cite{JSM}  a variation of Bowen's topological entropy that can be applied to the study of discontinuous semiflows on compact metric spaces was introduced. The main novelty is the use of certain family of pseudosemimetrics associated to the semiflow, in a such a way that in the continuous case they coincide with the  classical Bowen's entropies.  It was proved that  a regular impulsive semiflow can be semiconjugated by a uniformly continuous bijection to certain continuous semiflow defined in a compact metric space.  In the cited paper, the entropy is defined by the growth rate of separated sets. The main results of the present work show that, for regular impulsive semiflows, the  entropy defined by using  spanning sets agrees with the entropy using separated sets. Moreover,  this  notion of entropy satisfy a variational principle. This results are obtained introducing a little change on the family of pseudosemimetrics associated to.
\end{abstract}

\maketitle

\section{Introduction}

The notion of metric entropy is  one of the most successful invariants in ergodic theory, and was introduced by Kolgomorov in 1958 and by Sinai in 1959 (see \cite{kolmogorov} and \cite{sinai}, respectively). In 1965 Adler, Konheim and McAndrew  (see \cite{alder}) gave the concept of entropy for continuous maps on compact topological space and proved that this notion of entropy is an invariant for continuous mappings. The relationship between metric entropy and topological entropy is given by the well-known variational principle. 

In 1970 Dinaburg (see \cite{dinaburg}) details a variational principle for a homeomorphism on a compact metric space of finite dimension. Goodman in 1971 (see \cite{goodman})  extend this result to a general continuous map on any compact Hausdorff space. Another two definitions of topological entropy for homeomorphism and flow on a not necessarily compact metric space was given by Bowen in 1971 (see \cite{bowen}). Bowen gave the notion using spanning sets and separated sets and proved that they  agree when the space is compact. The variational principle for continuous semiflows on compact metric spaces is obtained with \cite[Theorem 2.1]{dinaburg} and \cite[Proposition 21]{bowen}. 

In the context of non necessarily continuous semiflows on a compact metric space, in 2015 Alves, Carvalho and V\'asquez (see \cite{v2})  introduced the notion of topological $\tau$-entropy, which depends on an admissible function $\tau$, and they proved the variational principle for impulsive semiflows. Later, in \cite{JSM} were introduced some variations of  Bowen's entropy definitions  for non necessarily continuous semiflows, by using certain family of semi-metrics associated to the semiflow.  Moreover, \cite[Theorem 2]{JSM} proved that this definitions are a lower bound for the $\tau$-entropy defined in \cite{v2}. In this work we consider  finer pseudosemimetrics, which generate finer entropy notions than those defined in \cite{JSM}. The main results of this paper establish that for regular impulsive dynamical systems   all this notions of entropy coincide and satisfy the variational principle. The paper is organized as follows: In the second section the necessary notation and  definitions are introduced, and the statements of main results are given. All the following sections are devoted  one by one  to prove each of the one  announced theorems.

\section{Setting and Statements}\label{section1}

Here and throughout this paper, we  denote by  $(X,d)$ a compact metric space and $\phi:\mathbb{R}^+_0\times X\to X$ a semiflow that is not necessarily continuous. We will use  the notation $\phi(t,x) = \phi_t(x)$. 
For $\delta>0$, we define the pseudosemimetric\footnote{A pseudosemimetric on $X$ is a map $D:X\times X\to \mathbb{R}_0^+$ satisfying $D(x,x)=0$ and $D(x,y)=D(y,x)$ for all $x,y\in X$.}
$d_{\delta}^{1,\phi}=d_{\delta}^{1}:X\times X\to\mathbb{R}_0^+$  as
\begin{eqnarray*}
d_{\delta}^{1}(x,y)=\inf\{d(\phi_{s_1}(x),\phi_{s_2}(y)):s_1,s_2\in[0,\delta)\}.
\end{eqnarray*}
Given  $\varepsilon>0$ and $T>0$, a set $F\subseteq X$ is said to be\emph{ $(\phi,T,\varepsilon,\delta)$-spanning} for the pseudosemimetric $d_{\delta}^1$, if for all $x\in X$, there exists $y\in F$ such that for all $t\in[0,T]$
$$d_{\delta}^1(\phi_t(x),\phi_t(y))<\varepsilon.$$
A set $E\subseteq X$ is said to be\emph{ $(\phi,T,\varepsilon,\delta)$-separated} for the pseudosemimetric $d_{\delta}^1$, if for all $x,y\in E$, $x\neq y$, there exists $t\in[0,T]$ such that 
$$d_{\delta}^1(\phi_t(x),\phi_t(y))\geq\varepsilon,$$
Define 
$$\bar{r}(\phi,T,\varepsilon,\delta)=\inf\{|F|: F \mbox{ is } (\phi,T,\varepsilon,\delta)\mbox{-spanning for }\, d_{\delta}^1\}$$
and
$$\bar{s}(\phi,T,\varepsilon,\delta)=\inf\{|E|: E \mbox{ is } (\phi,T,\varepsilon,\delta)\mbox{-separated for } \, d_{\delta}^1\}$$
Here, $|A|$ denotes the cardinality of a subset $A\subseteq X$. The \emph{growth rate} of $\bar{r}(\phi,T,\varepsilon,\delta)$ and $\bar{s}(\phi,T,\varepsilon,\delta)$ are defined as

$$\bar{h}_{\text{r}}(\phi,\varepsilon,\delta)=\limsup_{T\to+\infty}\frac{1}{T}\log \bar{r}(\phi,T,\varepsilon,\delta)$$
and 
$$\bar{h}_{\text{s}}(\phi,\varepsilon,\delta)=\limsup_{T\to+\infty}\frac{1}{T}\log \bar{s}(\phi,T,\varepsilon,\delta),$$
where $\log\infty=\infty$. 
Since the function $\varepsilon \to \bar{h}_{\text{r}}(\phi,\varepsilon,\delta)$ and $\varepsilon \to \bar{h}_{\text{s}}(\phi,\varepsilon,\delta)$ are nonincreasing, we  define
$$\bar{h}_{\text{r}}(\phi,\delta)=\lim_{\varepsilon\to0^+}\bar{h}_{\text{r}}(\phi,\varepsilon,\delta)$$
and
$$\bar{h}_{\text{s}}(\phi,\delta)=\lim_{\varepsilon\to0^+}\bar{h}_{\text{s}}(\phi,\varepsilon,\delta),$$
In the same way, as the function $\delta \to \bar{h}_{\text{r}}(\phi,\delta)$ and $\delta \to \bar{h}_{\text{s}}(\phi,\delta)$ are also nonincreasing, we define 
\begin{eqnarray*}
\bar{h}_{\text{r}}(\phi)=\lim_{\delta\to0^+}\bar{h}_{\text{r}}(\phi,\delta) 
\end{eqnarray*}
\begin{eqnarray*}
\bar{h}_{\text{s}}(\phi)=\lim_{\delta\to0^+}\bar{h}_{\text{s}}(\phi,\delta) 
\end{eqnarray*}
A first result show that for a continuous semiflow $\phi$, we recover Bowen's classical entropy $h_{top}(\phi)$ (see \cite{bowen}). 
\begin{maintheorem}\label{teo0}
Let $\phi:\mathbb{R}_0^+\times X\to X$ be a continuous semiflow on the compact metric space $(X,d)$. Then
$$\bar{h}_{\text{r}}(\phi)=\bar{h}_{\text{s}}(\phi)=h_{top}(\phi).$$
\end{maintheorem}
In \cite{JSM} was consider the entropy for $\phi$ by using the coarse pseudosemimetric $d_{\delta}$. The entropy by  separated sets  and the entropy by spanning sets were denoted by $\hat{h}_{\text{s}}(\phi)$ $\hat{h}_{\text{r}}(\phi)$, respectively. Moreover, these entropies are related with  $ h_{\text{top}}^{\tau}(\phi)$, the $\tau$- entropy for an  admissible function $\tau$ on $X$ as defined in \cite{v2}. Furthermore, it was proved that
$$\hat{h}_{\text{r}}(\phi)\leq\hat{h}_{\text{s}}(\phi) \leq h_{\text{top}}^{\tau}(\phi).$$  As a second result, we have the following theorem.
\begin{maintheorem}\label{teo1}
Let $\phi:\mathbb{R}_0^+\times X\to X$ be a semiflow on the compact metric space $(X,d)$. Then
$$\bar{h}_{\text{r}}(\phi)\leq \bar{h}_{\text{s}}(\phi), \,\,\, \bar{h}_{\text{r}}(\phi)\leq\hat{h}_{\text{r}}(\phi) \,\, \text{ and } \,\,\bar{h}_{\text{s}}(\phi)\leq \hat{h}_{\text{s}}(\phi).$$

\end{maintheorem}

Next, we will consider  these kind of entropies in the context of  regular impulsive semiflows. In order to state our following results, we must introduce the impulsive dynamical systems (for more details see \cite{v1}, \cite{i7}, \cite{i8}, \cite{i9}, \cite{i5} and \cite{i0}).\\
Let   
$$\varphi:\mathbb{R}^+_0\times X\to X$$ be a continuous semiflow on $X$, $D\subset X$ a proper closed subset  and  $I:D\to X$ a continuous function. The \emph{first impulse time function} $\tau_1:X\to\mathbb{R}_0^+\cup\{\infty\}$  is defined by
$$\tau_1(x)=\left\{\begin{array}{ll}\inf\left\{t> 0:\varphi_t(x)\in D\right\} ,& \text{if } \varphi_t(x)\in D\text{ for some }t>0;\\
                                                                       +\infty, & \text{otherwise.}\end{array}\right.$$
If $\tau_1(x)<\infty$, we define the \emph{first impulse point} as
$$x^1=x^1(x)=\varphi_{\tau_1(x)}(x).$$
Inductively, if we have defined $x_n$ and $\tau_n(x)$, and $\tau_n(x)<\infty$ then  the \emph{$(n+1)$-th impulsive time} and \emph{$(n+1)$-th impulsive point} are defined,   respectively, by
$$\tau_{n+1}(x)= \tau_1(x) +\sum_{k=1}^{n}\tau_1(I(x^k)),$$
and 
$$x^{n+1}=x^{n+1}(x)=\varphi_{\tau_1(I(x^n))}(I(x^n)), \text{ if } \tau_1(I(x^n))<\infty.$$
 
Let $T(x)= \sup \{\tau_n(x):n\geq 1\}$. 
The \emph{impulsive drift} $\gamma_x:[0,T(x))\to X$ for a point $x\in X$  is defined inductively 
by
$$\gamma_x(t)=\varphi_t(x), \mbox{ if } t \in [0, \tau_1(x)),$$
and
$$\gamma_x(t)=\varphi_{t-\tau_n(x)}(I(x^n)), \mbox{ if }  t \in [\tau_n(x), \tau_{n+1}(x)) \mbox{ and }\tau_n(x)<\infty.$$
Observe that if $I(D)\cap D=\emptyset$ then $T(x)=\infty$. 

\begin{defi}We say that $(X,\varphi, D, I)$ is an \emph{impulsive dynamical system} if for all $x\in X$ we have:
 \begin{enumerate}
 \item $\tau_1(x)>0$ and
 \item $T(x)=\infty$.
 \end{enumerate}
 \end{defi}
For each impulsive dynamical system $(X,\varphi,D,I)$, we define the  associated \emph{ impulsive semiflow} $\phi: \mathbb{R}_0^+\times X \rightarrow X$  as
$$ \phi_t(x) = \gamma_x(t),$$
where $\gamma_x(t)$ is the impulsive drift for $x\in X$. It is easy to see that $\phi$ is indeed a semiflow (see \cite{i7}), although it is not necessarily continuous. Moreover, when $(X,\varphi, D, I)$ is an impulsive dynamical system, then $\tau_1$ is lower semicontinuous on $X\setminus D$ (see \cite{i5}). \\
For $\eta>0$, we put
$$\begin{array}{rcl}D_{\eta}=\bigcup_{x\in D}\{\varphi_t(x): 0<t<\eta\} &\mbox{and}& X_{\eta}=X\setminus (D\cup D_{\eta}).\end{array}$$
In \cite[Definition 2.1]{JSM} was defined regular impulsive dynamical system. In this work we need to be more specific in the property (2) of this definition.
\begin{defi}\label{RIDS}We say that the impulsive dynamical system $(X,\varphi, D, I)$ is \emph{regular} if $I(D)\cap D=\emptyset$, $I$ is Lipschitz and there exists $\eta>0$ such that
\begin{enumerate}
\item $D_{\xi}$ is open for some  $0<\xi<\eta/4$,
\item $\varphi_{\xi}(D_{\xi})\subset X_{\xi}\setminus \varphi_{\xi}(D)$, and
\item for all $x\in I(D)$ and $t\in(0,\xi]$, $\varphi_t(x)\notin I(D)$.
\end{enumerate}\end{defi}

Note that from definition $D \cap \varphi_\xi (D) =\emptyset$.

The associated impulsive semiflow $\phi: \mathbb{R}_0^+\times X \rightarrow X$ for a the regular impulsive dynamical system  $(X,\varphi, D, I)$ is called  \emph{regular impulsive semiflow}. 
\begin{maintheorem}\label{teo2}
Let $\phi:\mathbb{R}_0^+\times X\to X$ be the regular semiflow associated to a regular impulsive dynamical system $(X,\varphi,D,I)$. Then
$$\bar{h}_{\text{r}}(\phi)=\bar{h}_{\text{s}}(\phi)=\hat{h}_{\text{r}}(\phi)=\hat{h}_{\text{s}}(\phi).$$
\end{maintheorem}
According this result, we can define \textit{the topological entropy for a regular impulsive semiflow} as $$\bar{h}_{top}(\phi):=\bar{h}_{\text{r}}(\phi).$$
The main result of this paper relates the topological entropy $\bar{h}_{\text{top}}(\phi)$ of a regular impulsive semiflow with the metric entropies $h_{\mu}(\phi_1)$, where $\phi_1$ is the time one map induced by $\phi$ (see \cite{kolmogorov}). In order to state our result, we must first give some definitions.\\
A map between two topological spaces is called \textit{measurable} if the pre-image of any Borel set is a Borel set. Notice that a semiflow $\phi: \mathbb{R}_0^+\times X \rightarrow X$ is measurable if for each $t\geq0$, $\phi_t:X\to X$ is measurable. Observe \cite[Proposition 5.1]{v1} proved that the semiflows asociated to an impulsive dynamical systems $(X,\varphi,D,I)$, are measurables.\\
A probability measure $\mu$ on the Borel sets of the topological space $X$ is said to be \textit{$\phi$-invariant} if $\phi$ is measurable and
$$\mu(\phi_t^{-1}(A))=\mu(A)$$
for  every Borel set $A\subset X$ and every $t\geq0$. We denote by $\mathcal{M}(X)$ the set of all probability
measures on the Borel sets of $X$ and by $\mathcal{M}_{\phi}(X)$ the set of $\phi$-invariant measures.\\
Given a measurable map $H:X\to Y$ between two topological spaces $X$ and $Y$,  the push-forward map
$H_*:\mathcal{M}(X)\to\mathcal{M}(Y); \mu \to H_*\mu$ is defined  by  $H_*\mu(A) = \mu(H^{-1}(A))$, for any Borel set $A\subset Y$. Observe that $\mu$ belongs to $\mathcal{M}_{\phi}(X)$ if and only if $(\phi_t)_*\mu=\mu$ for all $t\geq0$.\\
Let $\mu$ and $\nu$ be two probability measures, invariant under measurable transformations $f:X\to X$ and $g:Y\to Y$, respectively. We say that the systems $(f,\mu)$ and $(g,\nu)$ are \textit{ergodically equivalent} if one can find measurable sets $M\subset X$ and $N\subset Y$ with $\mu(X\setminus N)=0$ and $\nu(Y\setminus N)=0$, and a measurable bijection $H:X\to Y$ with measurable inverse, such that
$$H_*\mu=\nu \,\,\, \text{ and } \,\,\, H\circ f= g\circ H.$$
The main result of this paper is the following.
\begin{maintheorem}\label{teo3}
\begin{upshape}
Let $\phi:\mathbb{R}_0^+\times X\to X$ be the regular impulsive semiflow associated to a regular impulsive dynamical system $(X,\varphi,D,I)$.
Then $ \mathcal{M}_{\phi}(X)\neq \emptyset$ and
$$\bar{h}_{top}(\phi)=\sup_{\mu\in \mathcal{M}_{\phi}(X)}\{h_{\mu}(\phi_1)\}$$
\end{upshape}
\end{maintheorem}

\section{Proof Theorem  \ref{teo0}}
Let us consider $\phi:\mathbb{R}_0^+\times X\to X$ a continuous semiflow on the compact metric space $(X,d)$. Before proving Theorem \ref{teo0}, we briefly recall Bowen's definition of topological entropy. Given $\varepsilon>0$ and $T>0$, a  set $F\subseteq X$ is said to be\emph{ $(\phi,T,\varepsilon)$-spanning} if for all $x\in X$, there exists $y\in F$ such that for all $t\in[0,T]$,
$d(\phi_t(x),\phi_t(y))<\varepsilon.$ A set $E\subseteq X$ is said to be\emph{ $(\phi,T,\varepsilon)$-separated} if for all $x,y\in E$, $x\neq y$, there exists $t\in[0,T]$ such that
$d(\phi_t(x),\phi_t(y))\geq\varepsilon.$ Puts
$$r(\phi,T,\varepsilon)=\inf\{|F|: F \mbox{ is } (\phi,T,\varepsilon)\mbox{-spanning}\}$$
and
$$ s(\phi,T,\varepsilon)=\sup\{|E|: E \mbox{ is } (\phi,T,\varepsilon)\mbox{-separated}\}.$$ 
Let us consider the topological entropies defined by Bowen as
$$h_\text{r}(\phi)=\lim_{\varepsilon\to0^+}\limsup_{T\to+\infty}\frac{1}{T}\log r(\phi,T,\varepsilon).$$
and
$$h_\text{s}(\phi)=\lim_{\varepsilon\to0^+}\limsup_{T\to+\infty}\frac{1}{T}\log s(\phi,T,\varepsilon).$$
In \cite{bowen}, Bowen showed that if $X$ is  compact and $\phi$ is continuous, then $h_{\text{r}}(\phi) = h_{\text{s}}(\phi).$
The topological entropy $h_{top}(\phi)$ is defined by these numbers.
Theorem \ref{teo0} will follow from the following two lemmas.

\begin{lema}\label{A1}
Let $\phi:\mathbb{R}_0^+\times X\to X$ be a semiflow (not necessarily continuous) on the compact metric space $(X,d)$. Then
\begin{eqnarray*}\label{111}
 \bar{h}_{\text{r}}(\phi)\leq h_{\text{r}}(\phi) \, \text{ and } \, \bar{h}_{\text{s}}(\phi)\leq h_{\text{s}}(\phi).
\end{eqnarray*}
\end{lema}

\begin{proof}
Let consider $\delta>0$, $\varepsilon>0$ and $T>0$. First, consider $F$ a $(\phi,T,\varepsilon)$-spanning set, we claim that $F$ is $(\phi,T,\varepsilon,\delta)$-spanning  for the pseudosemimetric $d_{\delta}^1$. Indeed, suppose that there exists $x\in X$ such that for all $y\in F$, there exists $t_0\in[0,T]$ such that $d_{\delta}^1(\phi_{t_0}(x),\phi_{t_0}(y))\geq \varepsilon$. Particularity $d(\phi_{t_0}(x),\phi_{t_0}(y))\geq \varepsilon,$ which is a contradiction.\\
Second, if $E$ is a $(\phi,T,\varepsilon,\delta)$-separated set for the pseudosemimetric $d_{\delta}^1$, we claim that $E$ is $(\phi,T,\varepsilon)$-separated. Indeed, suppose that there exist $x,y\in E$, $x\neq y$ such that for all $t\in[0,T]$,
$d(\phi_{t}(x),\phi_{t}(y))< \varepsilon$. Particularity $d_{\delta}^1(\phi_{t}(x),\phi_{t}(y))< \varepsilon,$
which is another contradiction. Therefore
$$\bar{r}(\phi,T,\varepsilon,\delta)\leq |F|$$
and
$$|E|\leq s(\phi,T,\varepsilon).$$
Hence
$$\frac{1}{T}\displaystyle\log \bar{r}(\phi,T,\varepsilon,\delta)\leq \frac{1}{T}\displaystyle\log r(\phi,T,\varepsilon)$$
and
$$\frac{1}{T}\displaystyle\log \bar{s}(\phi,T,\varepsilon,\delta)\leq \frac{1}{T}\displaystyle\log s(\phi,T,\varepsilon).$$
Taking limits, we obtain the inequalities in the Lemma. 
\end{proof}

\begin{lema}\label{A2}
Let $\phi:\mathbb{R}_0^+\times X\to X$ be a continuous semiflow on the compact metric space $(X,d)$. Then
\begin{eqnarray*}\label{1}
\bar{h}_{\text{r}}(\phi)\geq h_{\text{r}}(\phi) \, \text{ and } \, \bar{h}_{\text{s}}(\phi)\geq h_{\text{s}}(\phi).\end{eqnarray*}
\end{lema}
\begin{proof}
Since $\phi$ is a continuous  and $X$ is a compact metric space, for all $\alpha>0$ there exists $\beta=\beta(\alpha)>0$ such that for all $z\in X$ and $t\geq 0$, we have
$$u\in [t,t+\beta]\Rightarrow d(\phi_t(z),\phi_u(z))<\alpha/4.$$
Let consider $\delta$, $\varepsilon$ and $T$ positves numbers, with $\delta<\beta$ and $\varepsilon<\alpha/2$. First, consider $F$ a $(\phi,T,\varepsilon,\delta)$-spanning set for the pseudosemimetric $d_{\delta}^1$, we claim that $F$ is $(\phi,T,\alpha)$-spanning. Indeed, suppose that there exists $x\in X$ such that for all $y\in F$, there exists $t_0\in[0,T]$ such that $d(\phi_{t_0}(x),\phi_{t_0}(y))\geq \alpha.$ By the triangle inequality, we have  
$$d(\phi_{t_0}(x),\phi_{t_0}(y))\leq d(\phi_{t_0}(x),\phi_u(x))+d(\phi_u(x),\phi_s(y))+d(\phi_s(y),\phi_{t_0}(y)),$$
and for all $u,s\in[t_0,t_0+\beta)$, we have
$$d(\phi_u(x),\phi_s(y))>\alpha/2.$$
Since $\delta<\beta$ and $\varepsilon<\alpha/2$, then
$$d_{\delta}^1(\phi_{t_0}(x),\phi_{t_0}(y))> \varepsilon,$$
which is a contradiction. 
Therefore
$$|F|\geq r(\phi,T,\alpha).$$
Second, if $E$ is a $(\phi,T,\alpha)$-separated set, we claim that $E$ is $(\phi,T,\varepsilon,\delta)$-separated for the pseudosemimetric $d_{\delta}^1$. Indeed, suppose that there exist $x,y\in E$, $x\neq y$ such that for all $t\in[0,T]$, $d_{\delta}^1(\phi_{t}(x),\phi_{t}(y))< \varepsilon.$ 
Again, by the triangle inequality, for all $u,s\in[t,t+\beta)$, we have 
$$d(\phi_{t}(x),\phi_{t}(y))\leq \alpha/2+d(\phi_u(x),\phi_s(y)).$$
Since $\delta<\beta$ and $\varepsilon<\alpha/2$, implies that for all $t\in[0,T]$
$$d(\phi_{t}(x),\phi_{t}(y))< \alpha,$$
which is a contradiction.
Therefore
$$\bar{s}(\phi,T,\varepsilon,\delta)\geq |E|.$$
Hence
$$\frac{1}{T}\displaystyle\log \bar{r}(\phi,T,\varepsilon,\delta)\geq \frac{1}{T}\displaystyle\log r(\phi,T,\alpha)$$
and
$$\frac{1}{T}\displaystyle\log \bar{s}(\phi,T,\varepsilon,\delta)\geq \frac{1}{T}\displaystyle\log s(\phi,T,\alpha).$$
Taking limits, we obtain the inequalities in the Lemma.
\end{proof}

\begin{proof}[ Proof of Theorem \ref{teo0}]
Put together the inequalities in Lemmas \ref{A1} and \ref{A2}, we get the result.
\end{proof}

\section{Proof Theorem \ref{teo1}}
Let us consider $\phi:\mathbb{R}_0^+\times X\to X$ a non necessarily continuous semiflow on the compact metric space $(X,d)$. Before proving Theorem \ref{teo1}, we briefly recall two definitions of topological entropy given in \cite{JSM}. For $\delta>0$, we define the pseudosemimetric $d_{\delta}^{\phi}=d_{\delta}:X\times X\to\mathbb{R}_0^+$  as
\begin{eqnarray*}\label{parametric}
d_{\delta}(x,y)=\inf\{d(\phi_s(x),\phi_s(y)):s\in[0,\delta)\}.
\end{eqnarray*} 
Given $\varepsilon>0$ and $T>0$, a set $F\subseteq X$ is said to be\emph{ $(\phi,T,\varepsilon,\delta)$-spanning} for the pseudosemimetric $d_{\delta}$ ,if for all $x\in X$ there exists $y\in F$ such that for all $t\in[0,T]$,
$d_{\delta}(\phi_t(x),\phi_t(y)<\varepsilon.$ 
A set $E\subseteq X$ is said to be\emph{ $(\phi,T,\varepsilon,\delta)$-separated} for the pseudosemimetric $d_{\delta}$, if for all $x,y\in E$, $x\neq y$, there exists $t\in[0,T]$ such that
$d_{\delta}(\phi_t(x),\phi_t(y))\geq\varepsilon.$
Puts
$$\hat{r}(\phi,T,\varepsilon,\delta)=\inf\{|F|: F \mbox{ is } (\phi,T,\varepsilon,\delta)\mbox{-spanning for } d_{\delta}\}.$$
and
$$ \hat{s}(\phi,T,\varepsilon,\delta)=\sup\{|E|: E \mbox{ is } (\phi,T,\varepsilon,\delta)\mbox{-separated for } d_{\delta}\}.$$ 
Define the entropies using spanning set and separated set, respectively, by
$$\hat{h}_\text{r}(\phi)=\lim_{\delta\to0^+}\lim_{\varepsilon\to0^+}\limsup_{T\to+\infty}\frac{1}{T}\log \hat{r}(\phi,T,\varepsilon,\delta).$$
and
$$\hat{h}_\text{s}(\phi)=\lim_{\delta\to0^+}\lim_{\varepsilon\to0^+}\limsup_{T\to+\infty}\frac{1}{T}\log \hat{s}(\phi,T,\varepsilon,\delta).$$
Recall that in \cite[Theorem 2]{JSM} was prove that the $\tau$- entropy as defined in \cite{v2},where $\tau$  is an  admissible function on $X$, satisfies  $$\hat{h}_{\text{r}}(\phi)\leq\hat{h}_{\text{s}}(\phi) \leq h_{\text{top}}^{\tau}(\phi).$$ 
Theorem \ref{teo1} will follow from the following lemmas.
\begin{lema}\label{lemacoro2}
Let $\phi:\mathbb{R}_0^+\times X\to X$ be a semiflow on the compact metric space $(X,d)$. Then
\begin{eqnarray*}
\bar{h}_{\text{r}}(\phi)\leq \bar{h}_{\text{s}}(\phi) .\end{eqnarray*}
\end{lema}
\begin{proof}
Let us consider $\delta>0$, $\varepsilon >0$ and $T>0$. Since the union of a partially ordered family (by set inclusion)  of separated sets is separated set, by Zorn's Lemma, we can take a maximal $(\phi, T,\varepsilon, \delta)$-separated set $E$ for $d_{\delta}^1$. We claim that $E$ is $(\phi, T,\varepsilon, \delta)$-spanning for the $d_{\delta}^1$. Indeed, suppose that  $E$ is not $(\phi,T,\varepsilon,\delta)$-spanning. Then there exists $x\in X$ such that for all $y\in E$, there exists $t\in[0,T]$ such that
$$d_{\delta}^1(\phi_t(x),\phi_t(y))\geq \varepsilon.$$
Therefore $E\cup\{x\}$ is a $(\phi,T,\varepsilon,\delta)$ separated set, which contradicts  the maximality condition of $E$.  This implies
$$\bar{r}(\phi,T,\varepsilon,\delta)\leq |E|\leq \bar{s}(\phi,T,\varepsilon,\delta),$$
hence
$$\frac{1}{T}\log \bar{r}(\phi,T,\varepsilon,\delta)\leq \frac{1}{T}\log \bar{s}(\phi,T,\varepsilon,\delta).$$
Taking limits, we obtain the inequality in Lemma.
\end{proof}
\begin{lema}\label{Lemacoro}
Let $\phi:\mathbb{R}_0^+\times X\to X$ be a semiflow on the compact metric space $(X,d)$. Then
\begin{eqnarray*}
\bar{h}_{\text{r}}(\phi)\leq \hat{h}_{\text{r}}(\phi) \,\, \text{ and }  \,\,
 \bar{h}_{\text{s}}(\phi) \leq \hat{h}_{\text{s}}(\phi).\end{eqnarray*}
\end{lema}
\begin{proof} It is direct consequence of
$d_{\delta}^1(x,y)\leq d_{\delta}(x,y)$ for all $x,y\in X$.

\end{proof}

\begin{proof}[ Proof of Theorem \ref{teo1}]
By Lemma \ref{lemacoro2} and Lemma \ref{Lemacoro}, we get the result.
\end{proof}

\section{Proof Theorem  \ref{teo2}}
The main result in \cite[Theorem 3]{JSM} proved that if $\phi$ is a regular impulsive semiflow, then there exist a compact metric space $Y$, a continuous semiflow $\tilde{\phi}:\mathbb{R}_0^+\times Y\to Y$ and a uniform continuous bijection $H:X_{\xi}\to Y$ such that for all $t\geq0$
$$\tilde{\phi}_t\circ H=H\circ\phi_t.$$
Moreover,
$$h_{top}(\tilde{\phi})=\hat{h}_{\text{s}}(\phi).$$

For the proof of Theorem \ref{teo2}, we need some technical lemmas. 
\begin{lema}\label{semiconj}
Let us consider $\phi$ and $\tilde{\phi}$  two semiflows on the metric spaces $(X,d)$ and $(\tilde{X},\tilde{d})$, respectively. If $H:X\to \tilde{X}$ is a uniformly continuous bijection such that for all $t\geq0$
\begin{eqnarray}\label{semiconjugation}
\tilde{\phi}_t\circ H= H\circ \phi_t,
\end{eqnarray}
then
$$ \bar{h}_{\text{r}}(\tilde{\phi})\leq \bar{h}_{\text{r}}(\phi).$$

\end{lema}
\begin{proof}
Since $H$ is uniformly continuous, for all $\varepsilon>0$ there exists $\beta(\varepsilon)=\beta>0$ such that for all $x,y\in X$, we have
$$d(x,y)<\beta \Rightarrow \tilde{d}(H(x),H(y))<\varepsilon.$$
Let us consider $\delta>0$, $T>0$ and  $F\subset X$ be a $(\phi,T,\beta,\delta)$-spanning set for $d_{\delta}^{1,\phi}$. We claim that $H(F)$ is  $(\tilde{\phi},T,\varepsilon,\delta)$-spanning for $d_{\delta}^{1,\tilde{\phi}}$. Indeed,  for all $\tilde{x}\in \tilde{X}$ there exists $y\in F$ such that for all $t\in[0,T]$ we have 
$d_{\delta}^{1,\phi}(\phi_t(H^{-1}(\tilde{x})),\phi_t(y))<\beta $. This implies $$\tilde{d}_{\delta}^{1,\phi}(H(\phi_t(H^{-1}(\tilde{x}))),H(\phi_t(y)))<\varepsilon,$$
and so, by \eqref{semiconjugation}, we deduce
$$\tilde{d}_{\delta}^{1,\tilde{\phi}}(\tilde{\phi}_t (\tilde{x}),\tilde{\phi}_t( H(y)))<\varepsilon.$$
This proves that $H(F)$ is $(\tilde{\phi},T,\varepsilon,\delta)$-spanning for $d_{\delta}^{1,\tilde{\phi}}$. Since
$|H(F)|=|F| $
we have 
$$ \bar{r}(\tilde{\phi},T,\varepsilon,\delta) \leq \bar{r}(\phi,T,\beta,\delta).$$
Taking logarithms and limits (noting that that $\beta\to0^+$ when $\varepsilon\to0^+$) we deduce the inequality in the Lemma.
\end{proof}
\begin{coro}\label{r1}
Let us consider  $\phi:\mathbb{R}^{+}_0\times X\to X$  the semiflow  associated to the regular impulsive system  $(X,\varphi,D,I)$ and $\xi$ in definition \ref{RIDS}. Then, there exist a compact metric space $Y$, a continuous semiflow $\tilde{\phi}:\mathbb{R}_0^+\times Y\to Y$ and a uniform continuous bijection $H:X_{\xi}\to Y$ such that 
\begin{equation}\label{sconj}
\tilde{\phi}_t\circ H=H\circ\phi_t, \mbox{ for all }t\geq 0
\end{equation}
Moreover,
 $h_{top}(\tilde{\phi})\leq \bar{h}_r(\phi|_{X_{\xi}}).$
\end{coro}
\begin{proof}
Let consider $X_{\xi}$ as in the Definition \ref{RIDS} and put $Y=\pi(X_{\xi})$. By properties of  regular impulsive systems we know that $Y$ is a compact metric space (see \cite[Proposition 5.4]{JSM}). Choose
$H= \pi|_{X_{\xi}} $ and $\tilde{\phi}$ the induced semiflow. These cleary satisfy \eqref{sconj}.
Now, by Theorem \ref{teo0} and Lemma \ref{semiconj}, we have
$$ \bar{h}_r(\tilde{\phi})= h_{top}(\tilde{\phi}) \leq \bar{h}_r(\phi|_{X_{\xi}}).$$
\end{proof}
\begin{mainpropo}
\label{r2}
 Let us consider $\phi:\mathbb{R}^{+}_0\times X\to X$  the semiflow  associated to the regular impulsive system  $(X,\varphi,D,I)$. Then
$$\bar{h}_r(\phi|_{X_{\xi}})\leq\bar{h}_r(\phi)$$

\end{mainpropo}
To prove  this proposition, we need established some properties about $D_{\xi}$ when $(X,\varphi,D,I)$ is a regular impulsive system.
\begin{lema}\label{af1} \textcolor{purple}{Let consider $(X,\varphi,D,I)$ a regular impulsive system. Then
$$\partial D_{\xi}=D\cup \varphi_{\xi}(D).$$}
\end{lema}
\begin{proof}
Given $x\in\partial D_{\xi}$, there are sequences $x_n\in D$ and $t_n\in (0,\xi)$ such that
$$\lim_{n\to\infty}\varphi_{t_n}(x_n)=x.$$ By compactness, suppose that $\lim_{n\to\infty}x_n=y\in D$ and $\lim_{n\to\infty}t_n=t\in[0,\xi]$. If $t\in(0,\xi)$, then $x=\varphi_t(y)\in D_{\xi}$, which is a contradiction, because $D_{\xi}$ is open. Therefore $x=\varphi_0(y)=y\in D$ or $x=\varphi_{\xi}(y)\in \varphi_{\xi}(D)$.
\end{proof}
\begin{lema}\label{V1V2}
Given  neighborhoods $V_1$ and $V_2$ of $D$ and $\varphi_{\xi}(D)$, respectively. There exists $\rho>0$ such that if $x\in X_{\xi}$, $y\in D_{\xi}$ and $d(x,y)<\rho$, then $x,y\in V_1$ or $x,y\in V_2$.
\end{lema}
\begin{proof}
Suppose that for all $n>0$, there exist $x_n\in X_{\xi}$ and $y_n\in D_{\xi}$ with $d(x,y)<1/n$ such that  $x_n\in V_1$ and  $y_n\in V_2$, or  $x_n\in V_2$ and  $y_n\in V_1$.  We can suppose that both sequences are convergent, say   $x_n \to x$ and   $y_n \to y$, and so $x=y$. This implies that $x \in \partial D_\xi = D\cup \varphi_{\xi}(D)\subset V_1\cup V_2$. Suppose that $x\in V_1$, then   $x_n \in  V_1$  and so  $y_n\in V_1$ for all large  $n$, which is a contradiction. On the other hand, to  suppose   $x\in V_2$ leads another contradiction. 
\end{proof}

\begin{lema}\label{af2}
Let consider $(X,\varphi,D,I)$ a regular impulsive system. For all $\theta\in(0,\xi)$, there exist $\theta_1 <\theta$ and  $\varepsilon_1>0$ such that $V_1=\{x\in X: d(x,D)<\varepsilon_1\}$ is neighborhood of $D$ that satisfy
\begin{enumerate}[a)]
 \item $\varphi_{\theta_1}(V_1)\subset D_{\xi}$
    \item $V_1=V_1^{out}\cup D\cup V_1^{in}$
    \item $V_1^{in}\subset D_{\xi}$
    \item $V_1^{out}\subset \overline{D_{\xi}}^c$
   
\end{enumerate}
\end{lema}
\begin{proof}
Let consider $\theta\in(0,\xi)$ and $x\in D$, then for all $\theta_1< \theta$, $\varphi_{\theta_1}(x)\in D_{\xi}$. Since $D_{\xi}$ is an open set and $D$ is compact, there exist  $\varepsilon_1>0$ such that $V_1=\{x\in X: d(x,D)<\varepsilon_1\}$ is a neighborhood of $D$ satisfying $\varphi_{\theta_1}(V_1)\subset D_{\xi}$. This prove a). For $V_1^{in}=V_1\cap D_{\xi}$ and $V_1^{out}=V_1\cap \overline{D_{\xi}}^c$, we have b), c) and d). 
\end{proof}
\begin{lema}\label{af3}
Let consider $(X,\varphi,D,I)$ a regular impulsive system. For all $\theta $, there exists $\theta_2 < \theta$ and $\varepsilon_2>0$ such that $V_2=\{x\in X: d(x,\varphi_{\xi}(D))<\varepsilon_2\}$ is neighborhood of $\varphi_{\xi}(D)$ that satisfy
\begin{enumerate}[a)]
     \item $\varphi_{t}(\varphi_{\xi}(D))\subset X_{\xi}\setminus \varphi_{\xi}(D)$ for all $t \in (0, \theta_2)$
     \item $\varphi_{\theta_2}(V_2)\subset X_{\xi}\setminus \varphi_{\xi}(D)$
    \item $V_2=V_2^{out}\cup \varphi_{\xi}(D)\cup V_2^{in}$
    \item $V_2^{in}\subset D_{\xi}$
    \item $V_2^{out}\subset \overline{D_{\xi}}^c$
   
\end{enumerate}
\end{lema}
\begin{proof}
Observe that since $\varphi_{\xi}(D_{\xi})\subset X_{\xi}\setminus \varphi_{\xi}(D)$, then for all $z\in \varphi_{\xi}(D)$ and for all $t\in(0,\xi)$, $\varphi_t(z)\in X_{\xi}\setminus \varphi_{\xi}(D)$. Indeed, suppose that there exists $z\in \varphi_{\xi}(D)$ and $t \in (0,\xi)$ such that $\varphi_{t}(z)\notin X_{\xi} \setminus\varphi_{\xi}(D)$. Take $x\in D$ such that $\varphi_{\xi}(x)=z$ therefore $\varphi_{\xi}(\varphi_{t}(x))=\varphi_{t}(z)\notin X_{\xi} \setminus\varphi_{\xi}(D)$. On the other hand, since $\varphi_{t}(x)\in D_{\xi}$ we have that $\varphi_{\xi}(\varphi_{t}(x))\in X_{\xi} \setminus\varphi_{\xi}(D)$, which is a contradiction. By the compactness of $\varphi_{\xi}(D)$, there exists a postive $\theta_2<\xi$ such that $\varphi_{t}(\varphi_{\xi}(D))\subset X_{\xi}\setminus \varphi_{\xi}(D)$ for all $t \in (0,\theta_2)$, which prove a). Moreover, by the continuity of $\varphi$ and since $ X_{\xi}\setminus \varphi_{\xi}(D)$ is an open set, there exists $\varepsilon_2>0$ such that $V_2=\{x\in X: d(x,\varphi_{\xi}(D))<\varepsilon_2\}$ is neighborhood of $\varphi_{\xi}(D)$ satisfying $\varphi_{\theta_2}(V_2)\subset X_{\xi}\setminus \varphi_{\xi}(D)$, this prove b) . For $V_2^{in}=V_2\cap D_{\xi}$ and $V_2^{out}=V_2\cap \overline{D_{\xi}}^c$, we have c), d) and e).
\end{proof}
 Let consider the neighborhoods $ V_1$ and $V_2$ given in Lemmas \ref{af2} and \ref{af3}  respectively. Shrinking $\epsilon_1$ and $\epsilon_2$ we can suppose $ V_1$ and $V_2$ are disjoint because $D \cap \varphi_\xi(D)=\emptyset$. Moreover, since $D\cap I(D)=\emptyset$, we can chose $\varepsilon_1>0$ and $\varepsilon_3>0$ such that  $d(V_1,V_3)\geq \alpha$ for some $\alpha>0$  where $V_3$   is neighborhood of $I(D)$ defined by $  V_3=\{x\in X: d(x,I(D))<\varepsilon_3\}$..
\begin{proof}[Proof of Proposition \ref{r2}]

Fix a small positive  $\delta$. Since  $\phi$ is a regular impulsive semiflow, for  $\theta_1\in(0,\xi)$, let consider  $\varepsilon_1$ and  $V_1$ given by the Lemma \ref{af2}. Moreover, let consider   $\theta_2>0$, $\varepsilon_2>0$ and  $V_2=\{x\in X: d(x,\varphi_{\xi}(D))<\varepsilon_2\}$ given by  the Lemma \ref{af3}. We can suppose that $\max\{\theta_1,\theta_2\} < \delta$. Let consider $\alpha>0$  and  $V_3$ as above. For any subset $F$ in $X$, denote by $F'$ the set given by   $F'=\{y_{t(y)}=\phi_{t(y)}(y):y\in F\}$, where for each $y\in F$
$$t(y)=\inf\{t\geq0: \phi_t(y)\in X_{\xi}\}.$$ 

Take $\varepsilon>0$, with  $\varepsilon<\min\{\alpha,\varepsilon_1,\varepsilon_2, \varepsilon_3\}$,  $T>3\delta$ and $F\subset X$ a $(\phi,T,\varepsilon,\delta)$-spanning set for $d_{\delta}^1$. 
We claim that $F'$ is a $(\phi|_{X_{\xi}},T,\varepsilon,3\delta)$-spanning set for $d_{3\delta}^1$, that is, for all $x\in X_{\xi}$, there exist $y'\in F'$  such that for all $t\in[0,T-2\delta]$ 
\begin{eqnarray}\label{0}
d_{3\delta}^1(\phi_t(x),\phi_t(y'))<\varepsilon.
\end{eqnarray}
Indeed, for all $x\in X_{\xi}$, there exists $y\in F$ such that  $ \forall t\in[0,T]$
$$d_{\delta}^1(\phi_t(x),\phi_t(y))<\varepsilon.$$
Particularly, $d_{\delta}^1(x,y)=\inf_{s_1,s_2\in[0,\delta)}d(\phi_{s_1}(x),\phi_{s_2}(y))<\varepsilon$. Then  there exists $s_1,s_2\in[0,\delta)$ such that 
$$d(\phi_{s_1}(x),\phi_{s_2}(y))\leq\varepsilon.$$
If $t(y)=0$, then $y\in X_{\xi}$, so taking $y'=y$ works.

On the other hand, if $t(y)>0$, then $y\in D$ or $y\in D_{\xi}$. 

Suppose that $y\in D$. 
  Since $\phi_{s_1}(x)\in X_{\xi}$ and $\phi_{s_2}(y)\in D_{\xi}$, by \ref{V1V2} we have $\phi_{s_1}(x)\in V_1^{out}$ and $\phi_{s_2}(y)\in V_1^{in}$,  because $\varepsilon<\varepsilon_1$. Then by Lemma \ref{af2} a)  $\varphi_{\theta_1}(\phi_{s_1}(x))\in D_\xi$, and so  $\phi_{\theta_1}(\phi_{s_1}(x))\in V_3$. Since $\phi_{\theta_1}(\phi_{s_1}(y))\in D_{\xi}$ and $\delta>\theta_1$, $d_{\delta}^1(\phi_{t'+s_1}(x),\phi_{t'+s_1}(y))\geq\varepsilon$, which is a contradiction. Therefore $y\notin D$.

 Next, suppose $y\in D_{\xi}$. 
We consider two cases depending on  $s_2\geq t(y)$ or $s_2<t(y)$.

First, if $s_2\geq t(y)$, taking $y'=y_{t(y)}$, we have $d_{3\delta}^1(x,y')\leq d_{\delta}^1(x,y')<\varepsilon$. Therefore it hold (\ref{0}). Second, if $s_2<t(y)$ then $\phi_{s_2}(y)\in D_{\xi}$. As  $\phi_{\theta_2}(x)\in X_{\xi}$ and $\varepsilon<\min\{\varepsilon_1,\varepsilon_2\}$,  by Lemma \ref{V1V2} we have $\phi_{s_1}(x)\in V_1^{out}$ and $\phi_{s_2}(y)\in V_1^{in}$ or $\phi_{s_1}(x)\in V_2^{out}$ and $\phi_{s_2}(y)\in V_2^{in}$. However, arguing as before the first alternative is discarded, and so $\phi_{s_1}(x)\in V_2^{out}$ and $\phi_{s_2}(y)\in V_2^{in}$. By Lemma \ref{af3} $\phi_{\theta_2}(\phi_{s_2}(y))\in X_{\xi}$. So $t(y)<2\delta$ because  $\theta_2 <\delta$. This implies  that for all $t\in[0,T-2\delta]$
$$d_{3\delta}^1(\phi_t(x),\phi_{t+t(y)}(y))\leq d_{\delta}^1(\phi_{t+t(y)}(x),\phi_{t+t(y)}(y))<\varepsilon.$$
Therefore  worth  (\ref{0}).  This prove the claim.

Then
$\bar{r}(\phi|_{X_{\xi}},T-2\delta,\varepsilon, 3\delta)\leq |F'|\leq |F|$
and therefore $$\bar{r}(\phi|_{X_{\xi}},T-2\delta,\varepsilon, 3\delta)\leq \bar{r}(\phi,T,\varepsilon,\delta)$$ from where
$$\frac{T-2\delta}{T}\cdot\frac{1}{T-2\delta}\displaystyle\log \bar{r}(\phi|_{X_{\xi}},T-2\delta,\varepsilon,3\delta)\leq \frac{1}{T}\displaystyle\log \bar{r}(\phi,T,\varepsilon,\delta).$$
Taking upper limit when $T\to \infty$, limits when $\varepsilon\to 0$ and $\delta\to0$ , we obtain the inequality in Proposition \ref{r2} and  end the proof.
\end{proof}

\begin{lema}\label{L31}
Let $\phi:\mathbb{R}_0^+\times X\to X$ be the semiflow of a regular impulsive dynamical system $(X,\varphi,D,I)$. Let consider $\tilde{\phi}$ and $H$ given by Corolario \ref{r1}. Then
$$h_{\text{top}}(\tilde{\phi})=\bar{h}_{\text{r}}(\phi).$$
\end{lema}
\begin{proof}
By \cite[Theorem 3]{JSM}, \cite[Theorem 2]{JSM} and Theorem \ref{teo1}, we have
$$h_{\text{top}}(\tilde{\phi})=\hat{h}_{\text{s}}(\phi)\geq \hat{h}_{\text{r}}(\phi)\geq\bar{h}_{\text{r}}(\phi).$$
By Corollary \ref{r1} and Lemma \ref{r2}, we have
$$h_{\text{top}}(\tilde{\phi})\leq \bar{h}_{\text{r}}(\phi|_{X_{\xi}})\leq \bar{h}_{\text{r}}(\phi).$$
With this, we obtain the result.
\end{proof}
\begin{lema}\label{L32}
Let $\phi:\mathbb{R}_0^+\times X\to X$ be the semiflow of a regular impulsive dynamical system $(X,\varphi,D,I)$. Let consider $\tilde{\phi}$ and $H$ given by Corolario \ref{r1}. Then
$$h_{\text{top}}(\tilde{\phi})=\bar{h}_{\text{s}}(\phi).$$
\end{lema}

\begin{proof}
By \cite[Theorem 3]{JSM} and Theorem \ref{teo1}, we have
$$h_{\text{top}}(\tilde{\phi})=\hat{h}_{\text{s}}(\phi)\geq\bar{h}_{\text{s}}(\phi).$$
By Corollary \ref{r1}, Proposition \ref{r2} and Theorem \ref{teo1}, we have
$$h_{\text{top}}(\tilde{\phi})\leq \bar{h}_{\text{r}}(\phi|_{X_{\xi}})\leq \bar{h}_{\text{r}}(\phi) \leq \bar{h}_{\text{s}}(\phi).$$
putting this last inequalities together we finish the proof.
\end{proof}
\begin{lema}\label{L33}
Let $\phi:\mathbb{R}_0^+\times X\to X$ be the semiflow of a regular impulsive dynamical system $(X,\varphi,D,I)$. Let consider $\tilde{\phi}$ and $H$ given by Corolario \ref{r1}. Then
$$h_{\text{top}}(\tilde{\phi})=\hat{h}_{\text{r}}(\phi).$$
\end{lema}

\begin{proof}
By \cite[Theorem 3]{JSM} and \cite[Theorem 2]{JSM}, we have
$$h_{\text{top}}(\tilde{\phi})=\hat{h}_{\text{s}}(\phi)\geq \hat{h}_{\text{r}}(\phi).$$
By Corollary \ref{r1}, Lemma \ref{r2} and Theorem \ref{teo1}, we have
$$h_{\text{top}}(\tilde{\phi})\leq \bar{h}_{\text{r}}(\phi|_{X_{\xi}})\leq \bar{h}_{\text{r}}(\phi)\leq \hat{h}_{\text{r}}(\phi) .$$
With this, we obtain the result.
\end{proof}
\begin{proof}[ Proof of Theorem \ref{teo2}]
By Lemma \ref{L31}, Lemma \ref{L32}, Lemma \ref{L33} and \cite[Theorem 3]{JSM}, we get the result.
\end{proof}
\section{A variational Principle}
In order to proof the variational principle for regular impulsive semiflows, we need some technical lemmas. First, observe that by Theorem \ref{teo2}, we may apply The Variational Principle (\cite{VO}) to $\tilde{\phi}$, getting
$$\bar{h}_{top}(\phi)=\sup_{\tilde{\mu}\in\mathcal{M}_{\tilde{\phi}}(\pi(X_{\xi}))}\{h_{\tilde{\mu}}(\tilde{\phi}_1)\}.$$
Now, we are due to connect the measure theoretical information of $\tilde{\phi}$ to the corresponding one of $\phi$, and  to prove that replacing $\pi(X_{\xi})$  and $\tilde{\phi}$ by $X$ and $\phi$ in the above equation, equality  remains valid.
\begin{lema}\label{medidacero}
Let $\phi:\mathbb{R}^{+}_0\times X\to X$ be the semiflow  associated to the regular impulsive system  $(X,\varphi,D,I)$. If $\mu$ is a $\phi$-invariant probability measure, then 
$$\mu(D)=0 \,\,\, \text{ and } \,\,\, \mu(D_{\xi})=0.$$
\end{lema}
\begin{proof}
First, assume that $\mu(D)>0$. Since $\phi$ is measurable, by Poincar\'e Recurrence Theorem for $\mu$ almost every $x\in D$ there exist infinitely many moments $t>0$ such that $\phi_t(x)\in D$. Since $I(D)\cap D=\emptyset$, for all $t>0$ $\phi_t(x)\notin D$, which is a contradiction. Hence, $\mu(D)=0$ \\
Now, assume that $\mu(D_{\xi})>0$, then for $\mu$ almost every $x\in D_{\xi}$ there exist infinitely many moments $t>\xi$ such that $\phi_t(x)\in D_{\xi}$. By Lemma 5.2 in \cite{JSM} , for each $t>\xi$, there exists $\tau$, with $\xi<\tau<t$, such that $\phi_{\tau}(x)\in D$. Since $I(D)\cap D=\emptyset$ we have $\phi_t(x)\notin D_{\xi}$, which is a contradiction. Hence, $\mu(D_{\xi})=0$.
\end{proof}
Consider the inclusion map $i:X_{\xi}\to X$. The following lemma, exchange ergodic data between $\phi$ and $\tilde{\phi}$. 
\begin{lema}\label{push1}
Let us consider $\phi:\mathbb{R}^{+}_0\times X\to X$  the semiflow  associated to the regular impulsive system  $(X,\varphi,D,I)$. Then there exist a compact metric space $Y$, a continuous semiflow $\tilde{\phi}:\mathbb{R}_0^+\times Y\to Y$ and a uniform continuous bijection $H:X_{\xi}\to Y$ such that 
\begin{equation}\label{sconj1}
\tilde{\phi}_t\circ H=H\circ \phi_t, \mbox{ for all }t\geq 0
\end{equation}
Moreover,
$$(i\circ H^{-1})_*:\mathcal{M}_{\tilde{\phi}}(Y)\to \mathcal{M}_{\phi}(X)$$
is well defined and is a bijection.
\end{lema}
\begin{proof}
By Corollary \ref{r1} (see Theorem 3 in \cite{JSM}), for $Y=\pi(X_{\xi})$, $H= \pi|_{X_{\xi}} $ and $\tilde{\phi}$ the induced semiflow  satisfy \eqref{sconj1}.\\
To prove that $(i\circ H^{-1})_*$ is well defined, let consider $\tilde{\mu}\in \mathcal{M}_{\tilde{\phi}}(\pi(X_{\xi}))$. Since $H^{-1}$ is well defined, we have that $H^{-1}_*\tilde{\mu} \in \mathcal{M}_{\phi}(X_{\xi})$. Moreover, this implies that $i_*(H^{-1}_*\tilde{\mu}) \in \mathcal{M}_{\phi}(X)$. As $i_*(H^{-1}_*\tilde{\mu})=(i\circ H^{-1})_*\tilde{\mu}$, we conclude that $(i\circ H^{-1})_*\tilde{\mu} \in \mathcal{M}_{\phi}(X)$. Therefore $(i\circ H)_*$ is well defined.\\
Since $H$ is a bijection and $i$ is injective, we have that $H_*$ is a bijection and $i_*$ is injective, respectively. Finally, only we need to prove that $i_*$ is surjective . Indeed, given $\mu\in \mathcal{M}_{\phi}(X)$, let consider the measure $\nu\in\mathcal{M}(X_{\xi})$ the restriction to $X_\xi$ of $\mu$, that is for any  borel subset $E\subset X_\xi$ then $$\nu(E)= \mu(E)$$
As $supp(\mu)\subset \Omega _{\phi}\subset X_{\xi}$, for each $t\geq0$ we have 
$$\begin{array}{ll} \nu(\phi_t^{-1}(E))                         &=\nu((\phi|_{X_{\xi}})_t^{-1}(E))\\
                                        & =\mu(\phi_t^{-1}(E)\cap X_{\xi})\\
                                        & = \mu(\phi_t^{-1}(E))\\
                                        & =\mu(E)\\
                                        & =\nu(E).\end{array}$$
Therefore $\nu\in\mathcal{M}_{\phi}(X_{\xi})$ and $i_*$ is surjective. With this, we obtain the desire result.
\end{proof}
\begin{proof} [Proof of Theorem \ref{teo3}]
From Lemma \ref{push1}, we have that $(\phi_1,\mu)$ and  $(\tilde{\phi}_1,\nu)$ are ergodically equivalent and $\mathcal{M}_{\phi}(X)\neq\emptyset$. Moreover, from Proposition 9.5.1 in \cite{VO} we have  $$\bar{h}_{top}(\phi)=\sup_{\mu\in\mathcal{M}_{\phi}(X_{\xi})}\{h_{\mu}(\phi_1)\}.$$
Finally, from Lemma \ref{medidacero}
$$\bar{h}_{top}(\phi)=\sup_{\mu\in\mathcal{M}_{\phi}(X)}\{h_{\mu}(\phi_1)\}.$$

\end{proof}

\end{document}